\documentclass[11pt,a4paper]{article}
\usepackage{amsmath, amssymb, amscd, amsthm, amsfonts}
\usepackage{graphicx}
\usepackage[numbers]{natbib}
\usepackage{hyperref}
\usepackage{wasysym}
\usepackage{comment}
\usepackage{tikz}
\usepackage{subcaption}
\usepackage{graphicx}
\usepackage[title]{appendix}

\newtheoremstyle{myplain}
  {3pt}{3pt}
  {\normalfont}
  {}
  {\bfseries}
  {.}
  {0.5em}
  {}
\numberwithin{equation}{section}

\theoremstyle{myplain}
\newtheorem{theorem}{Theorem}[section]
\newtheorem{lemma}[theorem]{Lemma}
\newtheorem{claim}[theorem]{Claim}

\newtheorem{problem}[theorem]{Problem}
\newtheorem{proposition}[theorem]{Proposition}

\newtheorem{corollary}[theorem]{Corollary}

\newtheorem{definition}[theorem]{Definition}

\newcommand{\EE}{\mathbb{E}}
\newcommand{\PP}{\mathbb{P}}

\newcommand{\cE}{\mathcal{E}}

\begin{document}
\title{Nearly tight bound for rainbow clique subdivisions in properly edge-colored graphs and applications}
\author{Peiru Kuang \footnote{School of Mathematical Sciences, Shanghai Jiao Tong University, Shanghai 200240, China. Email: peiru\_k@sjtu.edu.cn} 
\and 
Yan Wang \footnote{School of Mathematical Sciences, Shanghai Jiao Tong University, Shanghai 200240, China. Supported by National Key R\&D Program of China under Grant No. 2022YFA1006400 and National Natural Science Foundation of China under Grant No. 12571376. Email: yan.w@sjtu.edu.cn (corresponding author).}}

\date{}

\maketitle
\begin{abstract}
An edge-colored graph is said to be rainbow if all its edges have distinct colors. In this paper, we study the rainbow analogue of a fundamental result of Mader [\emph{Math. Ann.} \textbf{174} (1967), 265--268] on the existence of subdivisions in graphs with large average degree. This is part of the study of rainbow analogues of classical Tur\'an problems, a framework systematically introduced by Keevash, Mubayi, Sudakov and Verstraëte [\emph{Combin. Probab. Comput.} \textbf{16} (2007), 109--126]. We prove that every properly edge-colored graph on $n$ vertices with average degree at least $t^2(\log n)^{1+o(1)}$ contains a rainbow subdivision of $K_t$. When $t$ is a constant, this bound is tight up to the $o(1)$ term. So it essentially resolves a question raised by Jiang, Methuku and Yepremyan [\emph{European J. Combin.} \textbf{110} (2023), 103675] on rainbow clique subdivisions, and also implies a result of Alon, Bucić, Sauermann, Zakharov and Zamir [\emph{Proc. Lond. Math. Soc.} \textbf{130} (2025), e70044] on rainbow cycles. In addition, we present several applications of our result to problems in additive combinatorics, number theory and coding theory.
\end{abstract}
\maketitle
\date{}
\section{Introduction} \label{s1}
Generally speaking, an extremal problem asks for the maximum or minimum possible value of a parameter of a large structure that guarantees the existence of prescribed substructures. Extremal graph theory studies such problems for graph structures (see \cite{Bollobas1995} for a survey), which are also used to represent the properties of other mathematical objects relevant to extremal problems. One of the central problems in this area can be described as follows. Given a forbidden graph $H$, determine ex$(n, H)$, the maximum number of edges in a graph on $n$ vertices that does not contain $H$ as a subgraph. In the case when $H$ is a complete graph $K_{r}$, the value ex$(n, K_{r})$ was determined in 1941 by Tur\'an. Later, this result was generalized to $H$-free graphs by Erd\H{o}s and Stone \cite{ES1946} and by Erd\H{o}s and Simonovits \cite{ES1966}. All of these works had a significant impact on the development of extremal graph theory.

It is natural to consider colors on the edges, as colors often encode structural or geometric information~\cite{M2011,MPS2008}. Readers can refer to \cite{Sudakov2024} for the study of extremal problems on edge-colorings. An edge coloring is called \textit{proper} if all the edges incident to each vertex have distinct colors. An edge-colored graph is called \textit{rainbow} if every color appears at most once. A rainbow subgraph in an edge-colored graph is a fundamental concept of study, with some important applications in additive combinatorics. For example, a subset \( A \) of an abelian group \( G \) is called a \textit{Sidon set} of order $k$ \cite{Lindstrom2000} if each element \( g \in G \) has at most one representation as a sum of \( k \) (not necessarily distinct) elements from \( A \) (up to reordering of elements in the summation). In fact, given a subset $A$ of an abelian group $G$, one can construct an edge-colored bipartite graph: The vertex parts \( X \) and \( Y \) are both copies of \( G \), and for \( x \in X \) and \( y \in Y \), $xy$ is an edge if and only if $x-y \in A$ and we assign the color \( x-y \) to the edge \( xy \). This yields a properly edge-colored graph, and if \( A \) is a Sidon set of order $k$, then the graph contains no rainbow cycle of length $2k$.

The study of rainbow subgraphs and other restricted subgraphs in edge-colored graphs can be traced back to the work of Euler \cite{Eular1782} on Latin squares. Another famous example is the Ramsey's Theorem \cite{Ramsey1930} (also see Erd\H{o}s and Rado \cite{ER1950}), which shows that any edge coloring of $K_n$ using two colors must contain a rainbow $K_m$, provided that $n$ is large relative to $m$. Motivated by this, Alon, Jiang, Miller and Pritikin \cite{AJMP2003} introduced the problem of finding a rainbow subgraph $H$ in a coloring of $K_n$ in which each color appears at most $m$ times at each vertex. A systematic study on rainbow subgraphs was initiated by Keevash, Mubayi, Sudakov and Verstraëte~\cite{KMSV2007}, where they introduced rainbow analogues of classical Tur\'an problems. For a fixed graph $H$, the rainbow Tur\'an number \( \text{ex}^*(n, H) \) is the maximum number of edges in a graph on $n$ vertices that has a proper edge coloring with no rainbow copy of $H$. Clearly, we have \( \text{ex}^*(n, H)\geq \text{ex}(n, H) \). For any non-bipartite $H$, they~\cite{KMSV2007} show that \( \text{ex}^*(n, H)= (1+o(1))\text{ex}(n, H) \). However, rainbow Tur\'an problems for general bipartite graphs remain challenging.

\subsection{Rainbow cycles}
Cycle is a structure of particular interest in graph theory \cite{BS1974}. 
Seeking certain dependencies among mathematical objects often corresponds to finding cycles in a graph. Such encoding allows one to apply graph theoretic results to these problems, which could be extremely powerful in some cases. Indeed, the Sidon set problem mentioned earlier can be reduced to finding a cycle with certain properties. In a properly edge-colored graph, it is natural to require that a cycle have distinct colors. The existence of rainbow cycles in properly edge-colored graphs was first systematically studied in 2007 by Keevash, Mubayi, Sudakov and Verstraëte \cite{KMSV2007}. They raised the following problem. 
\begin{problem}[\cite{KMSV2007}]\label{problem}
How many edges are needed in an \(n\)-vertex graph to guarantee that every properly edge-colored graph contains a rainbow cycle?
\end{problem}
The non-rainbow version of this problem is elementary. Indeed, a graph must have at least \(n\) edges to guarantee the existence of a cycle. The rainbow version, however, has proven to be significantly more challenging. Cayley sum graphs together with their canonical edge coloring shows that a linear number of edges is no longer enough to guarantee such a cycle.
\begin{definition}
Given a group $\Gamma$ and a subset $S \subseteq \Gamma$, the \textit{Cayley sum graph} $Cay(\Gamma, S)$ is the graph with its vertex set being the elements of $\Gamma$ such that two vertices $x$, $y$ are adjacent if and only if $x + y = s$ for some $s \in S$. Moreover, $G$ has a canonical proper edge coloring, i.e. an edge $(x, y)$ has color $s = x + y$.
\end{definition}
Clearly, $Cay(\mathbb{F}_2^k, \{e_1,\ldots,e_k\})$ (also known as a $k$-dimensional hypercube) gives a lower bound of $\Omega(n \log n)$ for Problem \ref{problem}, where \( e_i \) is the vector with \( 1 \) in its \( i \)-th coordinate and \( 0 \) elsewhere. It was conjectured in \cite{KMSV2007} that this bound is also tight.

A series of improvements have been made to establish the upper bound for Problem \ref{problem}. Das, Lee and Sudakov~\cite{DLS2013} obtained the first nontrivial result, showing an upper bound of the form \(ne^{(\log n)^{1/2 + o(1)}}\). Their key approach is passing to an expander subgraph to find a rainbow cycle. Expander graphs are sparse graphs with strong connectivity properties, which have many applications in mathematics and computer science (see \cite{HLW2006} for details). Another important approach, based on homomorphism counting, was pioneered by Janzer~\cite{Janzer2022}, who subsequently improved this to \(O(n\log^4 n)\). This was later improved by Tomon~\cite{Tomon2024} to a bound of \(n(\log n)^{2 + o(1)}\) using sprinkling method. The \(o(1)\) term was removed independently by Janzer and Sudakov~\cite{JS2024}, and by Kim, Lee, Liu and Tran~\cite{KLLT2024}, who establish a bound of \(O(n\log^2 n)\). Very recently, Alon, Bucić, Sauermann, Zakharov and Zamir~\cite{Alon2025} obtained an essentially tight bound of $n(\log n)^{1+o(1)}$ by a refined probabilistic argument.

\subsection{Rainbow subdivisions}
A subdivision of a graph $G$, denoted by $TG$, is a graph obtained from $G$ by replacing each of its edges with internally vertex-disjoint paths. Subdivisions play a central role in topological graph theory. In the 1930s, Kuratowski~\cite{Ku1930} showed that a graph is not planar if and only if it contains a $TK_5$ or a $TK_{3,3}$. A fundamental result of Mader~\cite{Mader1967} from 1967 states that if \( G \) is a graph on \( n \) vertices that contains no \( TK_t \), then \( G \) has \( O_t(n) \) edges. Bollobás and Thomason~\cite{BT1998}, and independently Komlós and Szemerédi~\cite{KS1994,KS1996}, further improved this bound to \( O(t^2 n) \). Since then, many variations of this problem have been considered. An interesting direction is to control the size of the forbidden subdivisions. Montgomery~\cite{M2015} proved that the same bound holds even if one forbids all subdivisions of \( K_t \) of size at most \( O_t(\log n) \). A rainbow variant of the forbidden subdivision problem was proposed by Jiang, Methuku and Yepremyan \cite{JMY2023}. 
\begin{problem}[\cite{JMY2023}]\label{prob2}
Given $t \ge 3$, what is the smallest constant $c$ such that for all sufficiently large $n$, if $G$ is a properly edge-colored graph on $n$ vertices with $\Omega(n(\log n)^c)$ edges then $G$ contains a rainbow $TK_t$? In particular, is $c = 1$?
\end{problem}
They \cite{JMY2023} gave an upper bound of the form $ne^{O(\sqrt{\log n})}$. This upper bound was subsequently improved to $n(\log n)^{53}$ in \cite{JLMY2021}, where the authors used the fact that random walks mix rapidly in expanders, as well as tools such as homomorphism counting. This was further improved to $n(\log n)^{6+o(1)}$ by Tomon~\cite{Tomon2024}, and subsequently to $n(\log n)^{2+o(1)}$ by the second author~\cite{Yan2024}, who used a sprinkling method in expanders. We note that the lower bound $\Omega(n\log n)$ for Problem~\ref{problem} also applies to Problem~\ref{prob2}.
In this paper, we prove a nearly tight bound for Problem \ref{prob2} in a strong sense.
\begin{theorem}\label{main result}
Let $n$ and $t$ be integers. Let $G$ be a properly edge-colored graph on $n$ vertices with $d(G)\geq C t^2\cdot \log n \cdot (\log \log n)^{6}$, where $C$ is a sufficiently large constant. Then $G$ contains a rainbow $TK_t$, where each edge of $K_t$ is replaced by a path of length at most $O(\log n\cdot \log \log n)$.
\end{theorem}
Note that Theorem \ref{main result} is a rainbow analogue of Montgomery's result, that is, every properly edge-colored graph on \( n \) vertices with average degree at least \(\Omega(t^2 \log^{1+o(1)} n) \) contains a \( K_{t} \)-subdivision consisting of \( O_{t}(\log^{1+o(1)} n) \) vertices. Note that $t$ could depend on $n$ in Theorem \ref{main result}. Moreover, if $t$ is a constant, then the following corollary of Theorem \ref{main result} gives an essentially tight upper bound for Problem \ref{problem} on rainbow cycles and Problem \ref{prob2} on rainbow clique subdivisions.
\begin{corollary}\label{coro1}
Let $n$ and $t$ be integers. Every properly edge-colored graph on \( n  \) vertices with average degree at least $\Omega( (\log n)^{1+o(1)} )$ contains a rainbow $TK_t$, and in particular, a rainbow cycle.
\end{corollary}

\subsection{Applications}
In this subsection, we discuss some applications of Theorem \ref{main result} in additive combinatorics, number theory and coding theory.
\subsubsection{Additive combinatorics}
A famous open conjecture of Erd\H{o}s, first stated about 90 years ago (see \cite{Bohman1996} for a survey), is that if all subset sums of an integer set \(\Lambda \subseteq [1, n]\) are pairwise distinct, then \(|\Lambda| \leq \log_2 n + O(1)\). In additive combinatorics and harmonic analysis, a subset $S$ of an abelian group, all of whose subset sums are pairwise distinct, is called \textit{dissociated}. Equivalently, the only linear combination of elements in $S$ with coefficients in $\{-1, 0, 1\}$ that equals $0$ is the one with all coefficients equal to $0$. For general groups, we have the following definition (see \cite{Alon2025,Sudakov2024}).
\begin{definition} \label{dis}
Let \((\Gamma, \cdot)\) be a group with identity \(e\). A subset \(S  \subseteq \Gamma\) is called \emph{dissociated}
if there is no solution to the equation \(g_1^{\varepsilon_1} \cdots g_m^{\varepsilon_m} = e\) with $m\geq 1$, \(\varepsilon_i = \pm 1\) for all $i\in [m]$, and distinct \(g_1, \dots, g_m \in S\). 
\end{definition}
In other words, Definition \ref{dis} says that there is no non-trivial linear relation between elements in $S$ with coefficients in $\{-1,1\}$. In fact, dissociated sets play a role in groups analogous to that of linearly independent sets in vector spaces and have found numerous applications in additive combinatorics \cite{TV2006} and harmonic analysis \cite{KS2015,Rudin1990}. 
The \emph{additive dimension} \(\dim A\) of a subset \(A \subseteq \Gamma\) is the size of the maximum dissociated subset of \(A\), defined by Schoen and Skhredov \cite{SD2016}. Moreover, all maximal dissociated subsets of a given set have comparable sizes \cite{LY2011}. Therefore, the study of $\dim A$ is relevant to many problems in additive combinatorics \cite{BK2012,Bourgain1999,Chang2002,Sanders2011}.

A subset \( A  \subseteq \Gamma \) has \emph{small doubling} if the cardinality of \( A \cdot A = \{aa' : a, a' \in A\} \) is of order \( O(|A|) \). The study of the structures of sets with small doubling in the abelian setting can be traced back to Freiman’s seminal work~\cite{Freiman1964} from 1964 and has been the subject of extensive study ever since, see \cite{GR2007, Sanders2013_2} and references therein. There has been a lot of work extending results about sets with small doubling to nonabelian groups, see \cite{Tao2008}. It is natural to expect that a set with small doubling can be well approximated by the subgroup generated by its maximum dissociated subset. As a corollary of Theorem \ref{main result}, we obtain a version of Sanders' result \cite{Sanders2010} for all groups with a slightly weaker bound. 
\begin{theorem}\label{small-doubling}
Let \(\Gamma\) be a group, and let \(A \subseteq \Gamma\) be a finite subset with \(|A \cdot A| \leq K |A|\) for some positive integer \(K\). Then \(\dim A \leq O(\log^{1+o(1)} |A|)\).
\end{theorem}
\noindent By double counting the rainbow paths in a properly edge-colored graph with average degree at least $\Omega(\log n)$, we obtain a cycle with an edge of unique color that has at most logarithmic length. This implies \(\dim A \leq O(\log |A|)\) when $\Gamma$ is an abelian group in Theorem \ref{small-doubling}, thereby recovering Sanders' result~\cite{Sanders2010}. However, in a nonabelian group, we need the cycle to be rainbow, which makes it slightly different from an abelian group. 

In the following, we bound the dimension of symmetric sets introduced in \cite{TV2006}, or in other words, sets of large values of convolution. Our result introduces a $\log \log$ error term compared to \cite{SY2011}. However, we have significantly simplified the proof of \cite{SY2011} using Theorem \ref{main result}. We first give the definition.
\begin{definition}
Let \( \Gamma \) be a finite abelian group. For two subsets \( A, B \subseteq \Gamma \), the \textit{convolution} of $A$ and $B$ is
\[
(A * B)(x) = \sum_{y \in \Gamma} \mathbf{1}_A(y) \mathbf{1}_B(x - y) = |\{ (a,b) \in A \times B : a + b = x \}|,
\]
where $\mathbf{1}_A$ and $\mathbf{1}_B$ are indicator functions defined on $\Gamma$.
\end{definition}
By Theorem \ref{main result}, we can show the following.
\begin{theorem}\label{convolution}
Let \( G \) be a finite abelian group, and let \( A, B \subseteq G \) be two sets. Let \( \sigma \geq 1 \) be a positive real number. Define
$S = \{ x \in G : (A * (-B))(x) \geq \sigma \}$.
Then
\[
\dim(S) \le \max\{ |A|, |B|\} \cdot \sigma^{-1} \cdot O(\log^{1+o(1)}(\max\{ |A|, |B|\})).
\]
\end{theorem}
\noindent An appropriate version of Chang’s theorem \cite{Sanders2013} implies that $\dim(S) \leq |A||B| \sigma^{-1} \log(\min\{|A|, |B|\})$. Note that Theorem \ref{convolution} is stronger than Chang's theorem when $\min\{|A|, |B|\}$ is large.
\subsubsection{Number theory}
A \textit{Sidon sequence} is a sequence of integers $a_1<a_2<\ldots$ with the property that the sums $a_i+a_j$ $(i\leq j)$ are distinct. Sidon sequences were first studied by Sidon in the 1930s in connection with his work in Fourier analysis \cite{Sidon1932, Sidon1935}, and then by Erd\H{o}s in the 1940s \cite{Erdos1941}. Later Babai and Sós \cite{BS1985} generalised such sets to arbitrary groups and called them Sidon sets. Sidon sets play an important role in coding theory and cryptography \cite{Nagy2025,OBryant2004}. The following definition is a generalization of Sidon sets, see \cite{Hajela1988,Lindstrom2000}.
\begin{definition}
Let \( h \geq 2, \, g \geq 1 \) be integers. A subset \( A \) of integers is called a \textit{\( B_h[g] \)-sequence} if for every positive integer \( m \), the equation
\[
m = x_1 + \cdots + x_h, \quad x_1 < \cdots < x_h, \quad x_i \in A
\]
has at most \( g \) distinct solutions.
\end{definition}
Let \( F_h(g, n) \) denote the maximum size of a \( B_h[g] \)-sequence contained in \([n]\). A well-known question due to Halberstam and Roth is to derive bounds for \( F_h(g, n) \) \cite{HR1966}. If \( A \) is a \( B_h[g] \)-sequence in \([n]\), then \( \binom{|A| }{h} \leq ghn \), which implies $F_h(g, n)=O(n^{1/h})$.
For \( g = 1 \) and \( h = 2 \), it is possible to take advantage of counting the differences \( x_i - x_j \) instead of the sums \( x_i + x_j \), because the differences are also distinct. In this way, Erd\H{o}s and Tur\'{a}n \cite{ET1941} observed that a result of Singer \cite{Singer1938} implies that \( F_2(1,n) > \sqrt{n} \) infinitely many times. They also proved that $F_2(1,n) \leq n^{1/2} + O(n^{1/4})$. Further improvements can be found in \cite{Cilleruelo2010,Chen1993,Graham1996,Jia1993,Lindstrom1969}. However, for \( g > 1 \), the situation is completely different because the same difference can appear many times, and little progress has been witnessed for \( g > 1 \). 

Considerable attention has been paid to the Sidon set in various groups, especially in $\mathbb{Z}_n$ (see \cite{BC1962,Chen1994,Jia1993,Ruzsa1993,Singer1938}). Let \( f_h(g,n) \) denote the maximum size of a \( B_h[g] \)-sequence contained in \(\mathbb{Z}_n\). A similar calculation gives $f_h(g, n)=O(n^{1/h})$.
Bose and Chowla \cite{BC1962} constructed $B_h[1]$-sequences of size $(1 + o(1)) n^{1/h}$, showing that $n^{1/h}$ is the correct order of magnitude. 
When $g>1$,  our understanding of this problem is abysmal.
In this paper, we obtain a dichotomy of the existence of $B_h[g]$-sequences in $\mathbb{Z}_n$: Either we obtain a good upper bound for $B_h[g]$-sequences for any $g>1$, or we have a certificate of small size when $g=1$ (i.e., a non-$B_{h_0}[1]$-sequence of small size). 
\begin{theorem}\label{Bhg}
Let $n$ be an integer. Let $h$ and $g$ be integers (possibly related to $n$). Let $B\subseteq \mathbb{Z}_n$ be a $B_h[g]$-sequence of maximum size. Then either $|B|\leq (\log n)^{1+o(1)}$ or there exist an even integer $h_0\leq \frac{1}{2}\log n \log\log n$ and a set $B'\subseteq B$ of distinct elements such that $|B'|=2h_0$ and $B'$ is not a $B_{h_0}[1]$-sequence.
\end{theorem}

\subsubsection{Coding theory}
A \textit{locally correctable code (LCC)} is an error correcting code that admits a local correction algorithm that can recover any symbol of the original codeword by querying only a small number of randomly chosen symbols from the received corrupted codeword. Local correction was first introduced for program checking \cite{BK1995}, and it is well-known that Reed–Muller codes are locally correctable via polynomial interpolation. Since then, LCCs have been a mainstay in complexity and algorithmic coding theory with many applications in incidence geometry \cite{Dvir2012}, additive combinatorics \cite{BDL2013} and the theory of block designs \cite{BIW2010}. A binary linear code $C$ of block length $n$ is simply a subspace of $\mathbb{F}^n_2$. If $\dim(C)=k$, then one refers to it as an $[n, k]$-code. In the following, we assume that $C$ is a binary linear $[n, k]$-code. More formally, we have the following definition (see \cite{AG2024}).
\begin{definition}
    Given a binary linear $[n, k]$-code \( C \subseteq \mathbb{F}_2^n \), we say that it is a \textit{\((r, \delta)\)-locally correctable code} (abbreviated \((r, \delta)\)-LCC) for \( r \in \mathbb{N} \) and \( \delta \in (0, 1) \) if the following holds: For any received word \( y \in \mathbb{F}_2^n \) there exists a randomized algorithm \( D^y \) with oracle access to \( y \) that takes an index \( i \in [n] \) as input and satisfies the following properties:  
    \begin{enumerate}
        \item[(i)] \( D^y(i) \) makes at most \( r \) queries to \( y \), and
        \item[(ii)] if there exists a codeword \( c \in C \) satisfying \( d(y, c) \leq \delta n \), then \( D^y(i) \) outputs \( c_i \) with probability at least \( 2/3 \).
    \end{enumerate}
\end{definition}
For \( r = 1 \), it has long been known that 1-LCCs do not exist \cite{KT2000}. For \( r = 2 \), one must have \( n \geq \exp(\Omega_q(k)) \) \cite{GKST2006,KW2004}, so the Hadamard code is optimal. For 3-LCCs, Alrabiah and Guruswami \cite{AG2024} used the nearly tight bounds on rainbow cycles (Theorem 1.1 in \cite{Alon2025}, also Corollary \ref{coro1} in our paper) to obtain nearly tight bounds on the dimension of $(3,\delta)$-LCC, i.e. $k\leq O(\delta^{-2}\log^2 n\cdot \log\log n)$. Modulo the $\log \log n$ factor, this settles the dimension versus block length trade-off of 3-LCCs. It is noteworthy that the full resolution to Problem \ref{problem} will yield a tight bound on the dimension of 3-LCCs.

\medskip  
\noindent \textbf{Notation.} For a graph \( G \), we denote by \( V(G) \) and 
\( E(G) \) its vertex and edge sets, respectively. Let $d(G)$ denote the average degree of $G$. For \( U \subseteq V(G) \), write $G[U]$ for the induced subgraph by \( U \). For an edge \( vv' \in E(G) \), we write \( \gamma(v, v') \) for the color of the edge \( vv' \). The neighborhood of a vertex subset \( U \subseteq V(G) \) is the set of vertices in \( V(G) \setminus U \) that are adjacent to a vertex in \( U \). For a subset \( F \subseteq E(G) \) and a vertex \( v \in V(G) \), we denote by \( \deg_F(v) \) the number of edges in \( F \) incident to \( v \). Furthermore, 
$G-F$ denotes the graph obtained from $G$ by deleting all edges in $F$, and \( N_{G-F}(U) \) denotes the neighborhood of a vertex subset \( U \subseteq V(G) \) in \( G - F \).

All logarithms are of base \( e \) unless otherwise specified. Let \( f : \mathbb{R}_{>0} \to \mathbb{R}_{>0} \) and \( g : \mathbb{R}_{>0} \to \mathbb{R}_{>0} \) be two functions, and we write \( f = O(g) \) or \( g = \Omega(f) \) if there exists an absolute constant \( C \) such that \( f(x) \leq Cg(x) \) for all \( x \in \mathbb{R}_{>0} \). We write \( f = \Theta(g) \) if we have both \( f = O(g) \) and \( g = O(f) \). We also write \( f = o(g) \) if \( f(x)/g(x) \to 0 \) as \( x \to \infty \).
In this paper, we assume $n$ is sufficiently large.

\section{Proof overview} \label{s2}
The proof of Theorem \ref{main result} begins with a reduction to the case where the underlying graph is both a robust sublinear expander and $\log$-maximal (see Definitions \ref{rse} and \ref{log-maximal}). Then we show that such a graph must contain a rainbow clique subdivision as in the following lemma. Our proof strategy is inspired by \cite{Alon2025}. However, their method cannot be applied directly here because we need to avoid a prescribed set $\phi_0$ of forbidden vertices when dealing with subdivisions. We can show that the diameter of $G$ defined in Lemma \ref{pass to expander} is around $O(\log n \log \log n)$, and thus we have $|\phi_0|=O(\frac{d(G)}{(\log\log n)^4})$. When expanding very small vertex sets, the key expansion lemma (Lemma 4.1 of \cite{Alon2025}) fails due to the large size of $\phi_0$. To overcome this issue, we treat the expansion of small sets separately, see Lemma \ref{small set}. 
\begin{lemma}\label{pass to expander}
Let \( G \) be a $\log $-maximal graph and a robust sublinear expander on \( n \) vertices with \( d(G) \geq 10^8\cdot t^2 \cdot \log n \cdot (\log \log n)^{6} \). Then every proper edge-coloring \(\gamma : E(G) \to C\) of \( G \) contains a rainbow $TK_t$, where each edge of $K_t$ is replaced by a path of length at most $O(\log n\cdot \log \log n)$. 
\end{lemma}

Now we sketch the proof of Lemma \ref{pass to expander}. Let \( v_1, v_2, \ldots, v_t \) be \( t \) distinct vertices in \( G \). Let \( K \subseteq \binom{t}{2} \) be a maximal collection of pairs such that there exists a family of pairwise internally disjoint rainbow paths \( \mathcal{P} = \{k \in K : P_k\} \) such that
\begin{itemize}
    \item[(i)] For each \(\{i, j\} \in K\), \( P_{\{i,j\}} \) is a rainbow path of length \( O(\log n \cdot \log \log n) \) from \( v_i \) to \( v_j \);
    
    \item[(ii)] No colors appear more than once in \(\{\gamma(e) : e \in P, P \in \mathcal{P}\}\).
\end{itemize}
If \( K = \binom{t}{2} \), then the graph formed by all the paths in \( \mathcal{P} \) is a desired rainbow \( TK_t \). Hence, we may assume that there exist distinct \( i, j \in [t] \) such that \( \mathcal{P} \) contains no such path from \( v_i \) to \( v_j \). Let $\phi_0$ and $\phi_1$ be the set of forbidden vertices and colors of size $O(t^2 \log n \log\log n)$.

Fix any vertex \( x \in V(G) \). A \textit{rainbow path} in \( G \) is a path on which all edges receive distinct colors.
For a color subset \( A \subseteq C \), let 
\[
RP_{\phi_1}(x,A):=\{v\in V(G)\mid \exists \text{ a rainbow path with colors in \( A\setminus\phi_1 \) from $x$ to $v$}\}.
\]
Let $A_0 \subseteq C$ be a color subset. Our goal is to show that the size of $RP_{\phi_1}(v_i,A_0)\setminus \phi_0 $ (resp. $RP_{\phi_1}(v_j,C \setminus A_0)\setminus \phi_0 $) is at least $(n+1)/2$. Thus, there exists a rainbow path from $v_i$ to $v_j$ by the pigeonhole principle. In fact, the core difficulty is to show that we indeed have large $RP_{\phi_1}(v_i,A_0)\setminus \phi_0$ with high probability (resp. $RP_{\phi_1}(v_j,C \setminus A_0)\setminus \phi_0 $). Our approach consists of two key steps. \medskip \\ 
\textbf{STEP 1. THINNING.} 
We thin \( A_0 \) to obtain a nested sequence of random color sets \( A_0 \supseteq A_1 \supseteq \cdots \supseteq A_{N} \), where each $A_{i+1}$ is obtained from $A_i$ by independently retaining each color with a fixed probability. This process is referred to as ``thinning''. For convenience, write $RP_l$ for $RP_{\phi_1}(v_i,A_l)\setminus \phi_0$, for $l\in [N]$. Note that we need the assumption that $|RP_j|$ is sufficiently large compared to $|\phi_0|+|\phi_1|\approx2|\phi_0|$ so that further expansion can be performed while avoiding certain vertices and colors. Intuitively, each set \(RP_{j-1}\) is obtained from \(RP_j\) by expanding through edges with colors in \(A_{j-1}\setminus A_j\). Furthermore, any vertex in $RP_{j-1} \setminus RP_j$ is adjacent to a vertex in $RP_j$ via an edge with a color in $A_{j-1} \setminus A_j$. A key difficulty is that we cannot guarantee a sufficient number of such vertices, even in expectation. However, since our graph $G$ is a robust sublinear expander, the neighborhood of $RP_j$ has a relatively large size. We distinguish two cases that may occur. \medskip \\
\noindent \textbf{Case 1. Many edges between $RP_j$ and $N(RP_j)$ have colors in $\overline{A_j}$.} In this case, we can expect a reasonable fraction of these edges to have colors in $A_{j-1}\setminus A_j$, implying that many vertices in $N(RP_j)$ have an edge into $RP_j$ with colors in $A_{j-1}\setminus A_j$. So $|RP_{j-1}|$ is quite a bit larger than $|RP_j|$ in expectation.\medskip \\
\noindent \textbf{Case 2. Many edges between $RP_j$ and $N(RP_j)$ have colors in $A_j$.} For an edge \( uv \) with color \( c \in A_j \setminus A_{j+1} \), $v \not \in RP_{j}$ and \( u \in {RP}_j \), the color \( c \) must appear on every rainbow path from \( v_i \) to \( u \) using colors in \( A_j \).
In this case, we can expect a reasonable fraction of these edges to have colors in $A_{j}\setminus A_{j+1}$, and thus many vertices in $N(RP_j)$ have an edge into $RP_j$ with colors in $A_{j}\setminus A_{j+1}$. So $|RP_{j+1}|$ is quite a bit smaller than $|RP_j|$ in expectation.\medskip \\ 
\noindent In both cases, we could obtain $|RP_{i-1}|$ is quite a bit larger than $|RP_{i+1}|$ in expectation, see Lemma \ref{expansion rate}. Iterating this argument, we can expect the size of $RP_0$ to be quite large.
\medskip \\ 
\textbf{STEP 2. SPRINKLING.} 
After $N$ thinning steps, a sufficient number of colors from the set $R= A_{N}$ survive with high probability. We now apply the inverse procedure, often referred to as ``sprinkling". Instead of sampling all colors in one step, we sample in many independent rounds, in each of which every color is included with a certain probability. We show that, with high probability, at least $t^2\log^3 n$ vertices can be reached from $v_i$ by a rainbow path using colors from $R$. This implies that $|RP_N|$ is relatively large compared to $2|\phi_0|$, allowing the thinning process to proceed in STEP 1.

\section{Robust sublinear expanders} \label{s3}
In our proof of Theorem \ref{main result}, we first reduce to the case where $G$ is a robust sublinear expander, as mentioned above. Very roughly speaking, expanders are graphs that have good connectivity, yet may be
quite sparse. Komlós and Szemerédi \cite{KS1994,KS1996} introduced a notion of expander in which any set $X$ expands by a sublinear factor, that is, $|N_G(X)| \geq \rho(|X|)|X|$. Haslegrave, Kim and Liu \cite{HKL2022} extended this notion to a robust one such that similar expansion occurs even after removing a relatively small set of edges. Motivated by \cite{Alon2025,HKL2022}, we give the following definition.
\begin{definition}\label{rse}
A graph \( G \) on \( n \geq 1 \) vertices is called a \textit{robust sublinear expander} if
\begin{itemize}
    \item[(i)] for every \( 0 \leq \varepsilon \leq 1 \) and every non-empty subset \( U \subseteq V(G) \) of size \( |U| \leq n^{1-\varepsilon} \), and
    \item[(ii)] for every subset \( F \subseteq E(G) \) of \( |F| \leq \frac{\varepsilon}{4} \cdot d(G) \cdot |U| \) edges, 
\end{itemize}
we have
\[
|N_{G-F}(U)| \geq \frac{\varepsilon}{4} \cdot |U|.
\]
\end{definition}
Motivated by \cite{KS1994,KS1996,Yan2024}, we consider $\log$-maximal graphs in the following definition.   
\begin{definition}\label{log-maximal}
A graph $G$ is called $\log $-maximal if for every subgraph $H$ of $G$, we have
\[
\frac{d(H)}{\log(|V(H)|)} \leq \frac{d(G)}{\log(|V(G)|)}.
\]
\end{definition}
An obvious, but highly useful property is that every graph contains a $\log$-maximal subgraph. Indeed,
the subgraph $H$ maximizing the quantity $d(H)/\log (|V(H)|)$ is $\log$-maximal. Komlós and Szemerédi \cite{KS1994,KS1996}  defined a subgraph $H$ as \textit{$t$-maximal} if it has the maximum average degree among all subgraphs of $G$.
Moreover, it is easy to show that any $\log$-maximal graph is $t$-maximal. It follows that many of the favorable properties of $t$-maximal graphs are also true for $\log$-maximal graphs, notably the property of having a large minimum degree (see the following Lemma \ref{min degree}). The next lemma shows that $\log$-maximal graphs are good vertex-expanders.
\begin{lemma} \label{connection}
Every $\log$-maximal graph $H$ is also a robust sublinear expander.
\end{lemma}
\begin{proof}
Let \( n = |V(H)| \geq 2 \) and \( d = d(H) > 0 \). Fix \( 0 \leq \varepsilon \leq 1 \), and let \( U \subseteq V(H) \) be a non-empty subset with \( |U| \leq n^{1-\varepsilon} \). Suppose, for contradiction, that for some edge set \( F \subseteq E(H) \) with \( |F| \leq \frac{\varepsilon}{4} \cdot d|U| \), we have \( |N_{H-F}(U)| < \frac{\varepsilon}{4} \cdot |U| \).

Define \( H_1 = H[U \cup N_{H-F}(U)] \) and \( H_2 = H[V(H) \setminus U] \), and let \( n_1 = |V(H_1)| < (1 + \varepsilon/4)|U| \) and \( n_2 = n - |U| \). Then,
\[
e(H) \leq e(H_1) + e(H_2) + |F|.
\]
Since \( d(H_i) \leq d \cdot \frac{\log n_i}{\log n} \) for \( i = 1, 2 \), we obtain
\[
\frac{dn}{2} \leq e(H) \leq \frac{d}{2\log n}\left( n_1 \log n_1 + n_2 \log n_2 \right) + \frac{\varepsilon}{4} \cdot d|U|.
\]
Rearranging, we have
\begin{equation} \label{eq1}
n \leq \frac{1}{\log n}\left( n_1 \log n_1 + n_2 \log n_2 \right) + \frac{\varepsilon}{2} \cdot |U|.
\end{equation}
Note that \( \log n_1 \leq (1 - \varepsilon)\log n + \log(1 + \varepsilon/4) \), and since \( \log(1 + \varepsilon/4) < \varepsilon/4 \), we have
\[
n_1 \cdot \frac{\log n_1}{\log n} \leq \left(1 + \frac{\varepsilon}{4}\right)|U| \cdot \left(1 - \varepsilon + \frac{\varepsilon}{4\log n}\right) < \left(1 + \frac{\varepsilon}{4}\right)|U| \cdot \left(1 - \frac{3\varepsilon}{4}\right),
\]
where the last inequality uses \( \varepsilon/4 \leq (\varepsilon/4)\log n \) for large \( n \). Also, \( n_2 \log n_2 / \log n \leq n - |U| \). Substituting these into \eqref{eq1} yields
\[
n < \left(1 + \frac{\varepsilon}{4}\right)\left(1 - \frac{3\varepsilon}{4}\right)|U| + (n - |U|) + \frac{\varepsilon}{2}|U| = n - \frac{\varepsilon^2}{16}|U| < n,
\]
a contradiction. \qedhere
\end{proof}
Note that the definition of a robust sublinear expander already guarantees the large minimum degree, i.e. $\delta(G)\geq d(G)/4$, by simply applying the expansion property to the sets of size one. Thus, one can view Lemma \ref{min degree} as a stronger version of the easy classical result that any graph of average degree $d$ contains a subgraph with minimum degree at least $d/2$. 
\begin{lemma} \label{min degree}
Let $G$ be a $\log $-maximal graph. Then $\delta(G)\geq \frac{d(G)}{2}$.
\end{lemma}
\begin{proof}
Let \(\delta\) be the minimum degree of \(G\), and let \(v\) be a vertex of degree \(\delta\). Let \(H :=G- v\). Then, by the definition of \(\log \)-maximal, we have \(d(H)/\log |V(H)| \leq d(G)/\log |V(G)|\). This implies
\[
\frac{d(G)n-2\delta}{(n-1)\log |V(H)|} \leq \frac{d(G)}{\log |V(G)|},
\]
which gives
$d(G)n-2\delta \leq (n-1)d(G)$ and thus \(\delta(G) \geq d(G)/2\).\qedhere
\end{proof}

Now we are ready to show our main theorem assuming Lemma \ref{pass to expander} is true.
\begin{proof}[Proof of Theorem \ref{main result} assuming Lemma \ref{pass to expander}]
Let \( G \) be a properly edge-colored graph on \( n \) vertices with average degree \( d(G) \geq 10^8 \cdot t^2 \cdot \log n \cdot (\log \log n)^{6} \). Let $H$ be a $\log$-maximal subgraph of $G$ and hence a robust sublinear expander by Lemma \ref{connection}. Thus,
\[
d(H) \geq   \frac{\log |V(H)|}{\log n} \cdot d(G) \geq \frac{\log |V(H)|}{\log n} \cdot 10^8 \cdot t^2 \cdot \log n \cdot (\log \log n)^{6} > 10^8 \cdot t^2\cdot \log |V(H)| \cdot (\log \log |V(H)|)^{6}.
\]
So, by Lemma \ref{pass to expander} (noting that \( |V(H)| \geq d(H) \geq \log \log n \)), the subgraph \( H \) contains a rainbow $TK_t$, and then \( G \) also contains a rainbow $TK_t$.
\end{proof}

\section{Colorful expansion of random samples} \label{s4}
In this section, we prove Lemma \ref{pass to expander}. Let \( \gamma : E(G) \to C \) be as in Lemma \ref{pass to expander}.
Define
$K = \lceil 10 \log \log n \rceil$, $L = 10^5 \cdot \lceil \log n \rceil$ and $T = KL$
be three integers. Note that \( T \leq  10^7 \cdot \log n \cdot \log \log n \leq d(G)/(12\log \log n) \). Let $\phi_0$, $\phi_1$, $v_i$ and $v_j$ be as defined in Section \ref{s2}. Let \( A_0 \subseteq C \) be a random subset of the color palette, obtained by including each color in \( C \) with probability \( 1/2 \). Recall that $RP_{\phi_1}(x,A)$ is the set of vertices in $G$ which can be reached by a rainbow path with colors in $A\setminus \phi_1$ from $x$. To prove Lemma \ref{pass to expander}, it suffices to show the following lemma. 
\begin{lemma}\label{intersection}
We have 
\[
\ \mathbb{P}\left[ |RP_{\phi_1}(v_i,A_0)\setminus \phi_0| \geq n/\sqrt{e} \right] \geq 11/21.     
\]
\end{lemma}
\noindent Indeed, applying Lemma \ref{intersection} to the complement \( C \setminus A_0 \) and $v_j$ (which has the same distribution as \( A_0 \) and $v_i$), we also have
\[
\mathbb{P} \left[ |RP_{\phi_1}(v_j,C \setminus A_0)\setminus \phi_0| \geq n/\sqrt{e} \right] \geq 11/21.
\]
Thus, with positive probability, we have \( |RP_{\phi_1}(v_i,A_0)\setminus \phi_0| + |RP_{\phi_1}(v_j,C \setminus A_0)\setminus \phi_0| \geq (2/\sqrt{e}) \cdot n > n + 1 \). By pigeonhole principle, there exists a vertex \( u \in V(G) \) with \( u \in RP_{\phi_1}(v_i,A_0) \cap RP_{\phi_1}(v_j,C_{ij} \setminus A_0)\setminus \phi_0 \). 
This implies that \( u \) can be reached from \( v_i \) by a rainbow path with colors in \( A_0  \) and also from $v_j$ by a rainbow path with colors in \( C \setminus A_0 \). Thus, there must be a rainbow path between $v_i$ and $v_j$ obtained by concatenating these rainbow paths, as desired.  \hfill $\blacksquare$

Now we construct \( A_0 \) to obtain a nested sequence of random color sets \( A_0 \supseteq A_1 \supseteq \cdots \supseteq A_{(K-1)L} \), where for \( i = 0, \dots, (K-1)L  \). Define 
\[
\mathbb{P} \left[ x\in A_{i+1} | x\in A_i \right]=1=\frac{1}{T}.
\]
For convenience, write $RP_l$ for $RP_{\phi_1}(v_i,A_l)\setminus \phi_0$, for $l\in\{0,1,\ldots,(K-1)L\}$. Define \( f : \{0, \dots, K-1\} \to \mathbb{R} \) such that
\[
f(k) = n \cdot \exp \left( -\frac{1}{2} \cdot 10^k \right).
\]
Note that we have \( f(0) = n/\sqrt{e} >(n+1)/2\). Moreover, we can show that $RP_{(k-1)L}$ is large with high probability.
\begin{lemma} \label{small set}
We have
\[
\mathbb{P}\left[ |RP_{(K-1)L}| \geq t^2\log ^3 n \right] \geq 6/7.
\]
\end{lemma}
We postpone the proof of Lemma \ref{small set} to Section \ref{s5}. First, suppose Lemma \ref{small set} holds. Then we have $\mathbb{P}\left[ |RP_{(K-1)L}| \geq f(K - 1) \right] \geq 6/7$ as $t^2\log ^3 n>\log ^3 n>f(K-1)$.
The following lemma states that, for given \( k = 1, \dots, K-1 \), if we have \( |RP_{kL}| \geq f(k) \), then we likely also have \( |RP_{(k-1)L}| \geq f(k - 1) \). 

\begin{lemma} \label{iterate}
Assume $|RP_{(K-1)L}|\geq f(K-1)$. Then, for every \( k \in \{1, \dots, K-1\} \), we have
\[
\mathbb{P}[|RP_{(k-1)L}| < f(k - 1) \text{ and } |RP_{kL}| \geq f(k)] \leq \frac{1}{4^k}.
\]
\end{lemma}
\noindent We postpone the proof of this lemma to Section \ref{s6}. Iterating this will then give us a lower bound for \( |RP_0| \) with high probability. Now we prove Lemma \ref{intersection}.
\begin{proof}[Proof of Lemma \ref{intersection} assume Lemma \ref{iterate}]
Let \( E_{K-1} := \{ |RP_{(K-1)L}| \geq f(K-1) \} \) and \( E_0 := \{ |RP_0| \geq f(0) \} \). Define the \textit{bad} event \( F_k \) for \( k = \{1, \dots, K-1 \}\) as
$F_k = \{ |RP_{kL}| \geq f(k) \} \cap \{ |RP_{(k-1)L}| < f(k-1) \}$.
By Lemma~\ref{iterate}, we have \( \mathbb{P}(F_k) \leq 1/4^k \). Note that \(E_{K-1} \cap \overline{E_0} \subseteq \bigcup_{k=1}^{K-1} F_k\). By a union bound, we have
\[
\mathbb{P}(E_{K-1} \cap \overline{E_0}) \leq \sum_{k=1}^{K-1} \mathbb{P}(F_k) \leq \sum_{k=1}^{K-1} \frac{1}{4^k} < \frac{1}{3}.
\]
Since \( \mathbb{P}(E_{K-1} \cap E_0) = \mathbb{P}(E_{K-1}) - \mathbb{P}(E_{K-1} \cap \overline{E_0}) \), it follows that
$\mathbb{P}(E_{K-1} \cap E_0) > 6/7 - 1/3 = 11/21$.
Hence, we have
\[
\mathbb{P} \left( |RP_0| \geq \frac{n}{\sqrt{e}} \right)=\mathbb{P}(E_0) \geq \mathbb{P}(E_{K-1} \cap E_0) > \frac{11}{21},
\]
as desired.
\end{proof}
We also need the following probabilistic lemmas. 
\begin{lemma}[Markov's inequality, see \cite{Alon2004}] \label{Markov}
Let $Y$ be a nonnegative random variable and $a > 0$. Then we have
\[
P(Y > a\mathbb{E}[Y]) < \frac{1}{a}.
\]
\end{lemma}
\begin{lemma}[Chernoff bound, see \cite{Alon2004}] \label{Chernoff}
Let $X_1, \dots, X_m$ be independent $\{0, 1\}$-valued random variables. Let $X = X_1 + \cdots + X_m$, and $\mu = \mathbb{E}[X]$. Then, for every $\delta \geq  0$, we have
\[
\mathbb{P}[X < (1 - \delta)\mu] < e^{-\delta^2 \mu / 2}
\quad \text{and} \quad 
\mathbb{P}[X > (1 + \delta)\mu] < e^{-\delta^2 \mu / (2 + \delta)}.
\]
\end{lemma}

\begin{lemma}[\cite{Alon2025}] \label{union}
Let $m$ be a positive integer, let $0 < p < 1/m$ and let $\mathcal{E}_1, \ldots, \mathcal{E}_m$ be independent events such that for $i = 1, \ldots, m$, the event $\mathcal{E}_i$ holds with probability $p_i \geq p$. Then, with probability at least $pm/2$, at least one of the events $\mathcal{E}_1, \ldots, \mathcal{E}_m$ holds.
\end{lemma}

\begin{lemma}[\cite{Alon2025}]\label{cond prob}
For a given set \( C \), some integer \( T \) and \( 1 - (1/T) \leq p < 1 \), consider a random sequence of nested subsets \( C \supseteq A_0 \supseteq A_1 \supseteq \cdots \supseteq A_T \) obtained as follows: Let \( A_0 \subseteq C \) be a random subset obtained by including each element of \( C \) independently with probability \( 1/2 \). For any outcome of \( A_i \) and for some \( i = 0, \ldots, T-1 \), we define \( A_{i+1} \) to be a random subset of \( A_i \) obtained by including each element of \( A_i \) into \( A_{i+1} \) with probability \( p \) independently. Then, both of the following two statements hold.
\begin{itemize}
    \item[(i)] For integers \( 0 \leq i \leq j \leq T \), let us consider any outcome of \( A_j \). Then, conditional on this outcome of \( A_j \), for each element \( x \in C \), with probability at least \((j-i) \cdot (1-p)/6\) we have \( x \in A_i \).
    \item[(ii)] For integers \( 0 \leq i < j \leq T \) with \( j-i+1 \leq T/2 \), let us consider any outcomes of \( A_i \) and \( A_j \). Then, conditional on these outcomes of \( A_i \) and \( A_j \), for each element \( x \in A_i \), with probability at least \((T/(j-i)) \cdot (1-p)/2\) we have \( x \in A_{j-1} \).
\end{itemize}
\end{lemma}
\section{Expansion of small sets: Sprinkling} \label{s5}
In this section, we prove Lemma \ref{small set}. We adopt the \textit{sprinkling} technique. Roughly speaking, we sample colors with certain probability in each round so that the final distribution of colors is the same after this process ends. For convenience, denote $R:=A_{(K-1)L}$. We need the following definition.
\begin{definition}
Let $G$ be a graph with a proper edge-coloring $c : E(G) \to R$. Let $\phi$ be a set of forbidden vertices and colors. For $X \subseteq V(G)$ and $Q \subseteq R$, the \textit{restricted external neighborhood} of $X$ in $G$ with respect to the colors in $Q$ is
\[
RN_{Q, \phi}(X) = \{ y \in V(G) \setminus X : \exists x \in X, xy \in E(G), c(xy) \in Q \setminus \phi, y \notin \phi \}.
\]
\end{definition}
The following proposition states that, with high probability, many colors remain available after round one and the first rainbow neighborhood expands appropriately.
\begin{proposition} \label{merged-claim}
Let $p_c = 1/2$, and let $Q_1 \subseteq R$ be a random subset where each color is included independently with probability $p_c$. If $|\phi| \leq d(G) /(\log \log n)^4$, then the following holds.
\begin{enumerate}
    \item[(i)] $\PP\left[|R| \geq \dfrac{d(G)}{t^2\log \log n}\right] \geq \PP\left[|R| \geq \dfrac{|C|}{t^2\log \log n}\right] \geq \dfrac{99}{100}$.
    \item[(ii)] $\PP\left[|RN_{Q_1,\phi}(v_i)| > t^2\log n\right] \geq \dfrac{99}{100}$.
\end{enumerate}
\end{proposition}

\begin{proof}
We first prove (i). Define $q=(1-1/T)^{(K-1)L}/2$. Since $d(G) \leq |C|$, the first inequality is immediate. For the second, note that $\EE[|R|] = |C| q$. By the Chernoff bound (Lemma~\ref{Chernoff}), we have $\PP\left[|R| <|C|/t^2\log \log n\right] < e^{-\log n} < 1/100$.
We now prove (ii). By Lemma~\ref{min degree}, we have $\EE[|RN_{Q_1,\phi_0}(v_i)|] \geq \delta(G) \cdot p_c \cdot q - |\phi_0| =\Omega(\frac{d(G)}{\log\log n})$. Again applying the Chernoff bound (Lemma~\ref{Chernoff}), $\PP\left[|RN_{Q_1,\phi_0}(v_i)| < t^2\log n\right] < e^{-\log n} < 1/100$. \qedhere
\end{proof}
We also need a lemma from \cite{Tomon2024}. Given a bipartite graph, if we randomly sample the colors and vertices, then the size of the neighborhood of these vertices is very unlikely to be much smaller than its expected value.
\begin{lemma}[\cite{Tomon2024}] \label{Tomom}
Let $p, p_c \in (0,1]$, and $\lambda > 1$. Let $G$ be a bipartite graph with vertex classes $A$ and $B$, and let $f : E(G) \to R$ be a proper edge coloring. Let $U \subseteq A$ be a random sample of vertices, each vertex included independently with probability $p$, and let $Q \subseteq R$ be a random sample of colors, each included independently with probability $p_c$. Let $\mu := \EE(|N_{Q}(U)|)$, and suppose that every vertex in $A$ has degree at most $\Delta$. If $\Delta + |A| < \frac{\mu}{32\lambda \log_2(\lambda (pp_c)^{-1})}$ then
\[
\PP\left( |N_{Q}(U)| \leq \frac{\mu}{64\lambda \log_2(\lambda (pp_c)^{-1})}\right) < 2 e^{-\lambda}.
\]
\end{lemma}
To prove Lemma \ref{small set}, we need to establish the following technical lemma. It states that if $G$ is $\log$-maximal, then every vertex set of moderate size expands well, even after forbidding some colors and vertices. 
\begin{lemma} \label{l6.4}
Let $t$ be as in Lemma \ref{pass to expander}.
Let $\Theta(1/(\log n\cdot \log \log n)) < p_c \leq 1$, $\lambda=100\cdot (\log \log n)^2$ and $n > 0$ be a sufficiently large integer. Let $G$ be a graph on $n$ vertices with a proper edge-coloring $c : E(G) \to R$ and $B \subseteq V(G)$ satisfying the following.
\begin{enumerate}
    \item[(i)] $G$ is log-maximal.  
    \item[(ii)] $d(G) \geq 10^8\cdot t^2 \cdot \log n \cdot (\log \log n)^{6}$.
    \item[(iii)] $|\phi|\leq \frac{d(G)}{(\log \log n)^2} $.
    \item[(iv)] $t^2\log n \leq |B| \leq t^2\log ^3 n$.  
\end{enumerate}
Let $Q \subseteq R$ be a random subset of colors such that each color is chosen with probability $p_c$ independently. Then with probability at least $1 - 2e^{- \lambda^{1/2}}$, we have  
\[
|RN_{Q, \phi}(B)| \geq \frac{1}{\log n}|B|.
\]
\end{lemma}

\begin{proof}
Let $A = N_G(B)$. Let $H$ be a bipartite graph with vertex classes $A$ and $B$ and edge set $E(H) = \{xy : x \in B, y \in A \setminus \phi, c(xy) \in R \setminus \phi\}$. Let $H_Q$ be the subgraph of $H$ whose edges are colored from $Q$. Let $\Delta = \lambda^{1/2} p_c^{-1}$. Let $S = \{v \in A : |N_H(v)| \geq \Delta\}$ and $T = A \setminus S$. We consider two cases based on the number of edges between $B$ and $T$ in $G$. \medskip \\
\textbf{Case 1.} \( e_G(B, T) \leq \frac{d(G)|B|}{4}  \). \medskip
\begin{claim}
\( |S| \geq \frac{2|B|}{\log n} \) 
\end{claim}
Suppose \( |S| < \frac{2|B|}{\log n} \). Let \( C = V(G) \setminus B \). Since \( E(G) = E(G[B \cup S]) \cup E(G[C]) \cup E(G[B, T]) \), we have $d(G[B \cup S])(|B| + |S|)/2 = e(G[B \cup S]) \geq e(G) - e(G[C]) - e_G(B, T)  = d(G)n/2 - d(G[C])|C|/2 - e_G(B, T)$.   
As \( G \) is log-maximal, we have
\[
\frac{d(G[B \cup S])}{\log(|B| + |S|)} \leq \frac{d(G)}{\log n}, \quad \text{and} \quad \frac{d(G[C])}{\log |C|} \leq \frac{d(G)}{\log n}.
\]
Hence,
\[
\frac{d(G)(|B| + |S|) \log(|B| + |S|)}{2 \log n} \geq \frac{d(G)n}{2} - \frac{d(G)|C| \log |C|}{2 \log n} - e_G(B, T).
\]
As \( e_G(B, T) \leq \frac{d(G)|B|}{4}  \), we obtain
\begin{align*}
&(|B| + |S|) \log(|B| + |S|) \geq n \log n - |C| \log |C| - \frac{2 \log n}{d(G)} e_G(B, T) \\
&\geq |B| \log n + |C| (\log n - \log |C|) - \frac{2 \log n}{d(G)} \cdot \frac{d(G)|B|}{4 } \geq \frac{1}{2}|B| \log n .
\end{align*}
Since \( |S| < \frac{2|B|}{\log n}  \), we have
\[
|S| \log\left(|B|+\frac{2|B|}{\log n}\right) \geq |B| \left(\frac{1}{2}\log n - \log\left(|B|+\frac{2|B|}{\log n}\right)\right) .
\]
Therefore,
\[
|S| \geq |B|\left(\frac{1}{2}\frac{\log n}{\log\left(|B|+\frac{2|B|}{\log n}\right) }-1\right)>\frac{2}{\log n}|B|,
\]
a contradiction.\medskip \hfill $\blacksquare$

Now let \( W = N_{H_Q}(B) \cap S \). For every vertex \( y \in S \), we have
\[
\PP(y \in W) = 1 - (1 - p_c)^{|N_H(y) \cap B|} \geq 1 - (1 - p_c)^{\lambda^{1/2} p_c^{-1}} \geq 1 - e^{-\lambda^{1/2}}.
\]
Thus, \( \EE[|W|] \geq |S| (1 - e^{-\lambda^{1/2}}) \) and so \( \EE[|S \setminus W|] \leq |S| e^{-\lambda^{1/2}} \). By Markov's inequality (see Lemma \ref{Markov}), we have
\[
\PP(|W| \leq |S|/2) = \PP\left(|S \setminus W| \geq |S|/2\right) \leq \frac{2 \EE[|S \setminus W|]}{|S|} \leq 2e^{-\lambda^{1/2}}.
\]
Hence, we have
\[
\PP\left(|W| \geq \frac{|B|}{\log n}\right) \geq \PP(|W| \geq |S|/2) \geq 1 - 2e^{-\lambda^{1/2}}.
\]
This concludes the proof of Case 1.\medskip \\
\textbf{Case 2.} \( e_G(B, T) > \frac{d(G)|B|}{4}  \).\medskip

By (iii), we have
\[
e_H(B, T) \geq e_G(B, T) - |\phi| >\frac{d(G)|B|}{5}.
\]

\noindent Let \( T_1 = \{v \in T : |N_H(v)| < 2p_c^{-1}\} \) and \( T_2 = T \setminus T_1 \). We consider two subcases. \medskip \\ 
\noindent
\textbf{Subcase 2.1.} \(e_H(B, T_1) \geq e_H(B, T_2)\). \medskip \\
Let \(W = N_{H_Q}(B) \cap T_1\). For \(y \in T_1\), we have
\[
\PP(y \in W) = 1 - (1 - p_c)^{|N_H(y) \cap B|} > 1 - \left(1 - \frac{p_c|N_H(y) \cap B|}{2}\right) = \frac{p_c|N_H(y) \cap B|}{2}.
\]
Then
\[
\EE[|W|] = \sum_{y \in T_1} \PP(y \in W) > \frac{p_c}{2} \sum_{y \in T_1} |N_H(y) \cap B| = \frac{p_c}{2} e_H(B, T_1) \geq \frac{p_c}{4} e_H(B, T) > \frac{d(G)p_c|B|}{20}.
\]
\medskip 
\noindent
\textbf{Subcase 2.2.} \(e_H(B, T_1) < e_H(B, T_2)\).\\
Let \( W = N_{H_Q}(B) \cap T_2 \). For every vertex \( y \in T_2 \), we have
\[
\PP(y \in W) = 1 - (1 - p_c)^{|N_H(y) \cap B|} > 1 - e^{-2} > 2/3.
\]
Thus,
\[
\EE[|W|] = \sum_{y \in T_2} \PP(y \in W) >\frac{2}{3}|T_2| \geq \frac{2e_H(B, T_2)}{3\Delta} \geq \frac{e_H(B, T)}{3\Delta} > \frac{d(G)p_c|B|}{15\lambda^{1/2}} .
\]
In both subcases, we have \( \EE[|N_{H_Q}(B)|] \geq \EE[|W|] \geq \frac{d(G)p_c|B|}{15\lambda^{1/2}}  \). We will apply Lemma \ref{Tomom} with \[ (p, p_c, \lambda, G, A, B, K) = (1, p_c, \lambda^{1/2}, H[B, T], B, T, \Delta). \] 
Since $p_c>\Theta(1/(\log n\cdot \log \log n)) $, $d(G) \geq 10^8\cdot t^4 \cdot \log n \cdot (\log \log n)^{6}$ and also $|B| \geq  t^2\log n$, we have
\[
\frac{\EE[|N_{H_Q}(B)|]}{32\lambda^{1/2} \log_2(\lambda^{1/2} p_c^{-1})} \geq \frac{d(G)p_c|B|}{480\lambda \log_2(\lambda^{1/2} p_c^{-1})}  \geq \lambda^{1/2} p_c^{-1} + |B|.
\]
Thus by Lemma \ref{Tomom}, we have
\[
\PP \left( |N_{H_Q}(B)| \leq \frac{\ln 2\cdot\EE[|N_{H_Q}(B)|]}{64\lambda^{1/2} \log(\lambda^{1/2} p_c^{-1})} \right) \leq 2e^{-\lambda^{1/2}}.
\]
Therefore, with probability at least \( 1 - 2e^{-\lambda^{1/2}} \), we have
\[
|N_{H_Q}(B)| > \frac{\ln 2 \cdot \EE[|N_{H_Q}(B)|]}{64\lambda^{1/2} \log(\lambda^{1/2} p_c^{-1})} > \frac{1}{\log n}|B|,
\]
 by (ii). This completes the proof of Case 2.
\end{proof}
\begin{proof}[Proof of Lemma \ref{small set}]
For each \( i \in [L]\setminus\{1\} \), let \( Q_i \) be a random sample of \( Q \) where each color is included independently with probability \( q_i>\Theta(1/(\log n\cdot \log \log n)) \). Let $\phi$ be a set of forbidden vertices and colors of size $O(t^2\log n \log \log n)$.
We recursively define \(  S_i, B_i \) as follows. Let
\(
S_1 = B_1 := RN_{Q_1,\phi}(v_i).
\)
For \( i \in [L]\setminus\{1\} \), let
\[
S_{i+1} := \left\{ y \in N(B_i) \setminus (B_i \cup \phi) : 
\exists x \in B_i \text{ s.t. } y \notin \phi,\ xy \in E(G),\ f(xy) \in Q_{i+1} \setminus \phi \right\},
\]
and set \( B_{i+1} := B_i \cup S_{i+1} \).

We now apply Lemma \ref{l6.4} to \( (G, B_i, q_i, \lambda)_{\ref{l6.4}} \) for every \( i \in [L]\setminus\{1\} \). In each step, we grow $B_i$ to $B_{i+1}$ by Lemma \ref{l6.4} until $|B_i| > t^2\log^3 n$. This process must stop in at most $L$ steps, where $L$ is defined in Section \ref{s4}. Thus, with probability at least $1-2Le^{-\lambda^{1/2}}>99/100$,
there exists a subset $B\subseteq RP_{(K-1)L}$ such that $|B|\geq t^2\log ^3 n$. Thus, by Proposition \ref{merged-claim}, with probability at least $6/7$, we have $|RP_{(K-1)L}|\geq  t^2\log ^3n $.
This completes the proof of Lemma \ref{small set}. 
\end{proof}

\section{Expansion of large sets: Thinning} \label{s6}
In this section, we prove Lemma \ref{iterate}. Here, we adopt the inverse process of sprinkling, namely the \textit{thinning} method. For some fix $k\in \{1,2,\ldots,K-1\}$, let \(\cE_{k-1}\) be the event that \(|RP_{(k-1)L}| < f(k-1)\), and note that $\cE_{k-1}$ is determined by the outcome of \(A_{(k-1)L}\). Define the event $F_k:=\cE_{k-1} \cap \overline{\cE_{k}}$. Our aim is to show
\begin{equation} \label{bad}
 \PP[F_k] \leq 1/4^k.
\end{equation}
First we have the following.
\begin{claim}\label{assumption}
We have the following.
\begin{itemize}
    \item[(i)] \(\PP[\cE_{k-1}] \geq 1/4^k\).
    \item[(ii)] \(\EE[|RP_{kL}|\mid \cE_{k-1}] \geq f(k)/4^k\). 
\end{itemize}
\end{claim}
\begin{proof}
If (i) or (ii) is false, then we have \eqref{bad} true directly.
\end{proof}
Let
\begin{equation} \label{epi}
\varepsilon = \frac{10^{k-1}}{2 \log n}
\end{equation}
so that \(f(k-1) = n^{1-\varepsilon}\). The main step to prove Lemma \ref{iterate} is to show the following lemma. Roughly speaking, this lemma states that when conditioning on the event \(\cE_{k-1}\), for each \(j = 1, \dots, L/2 - 1\), we expect \(|RP_{kL-j-1}|\) to be much larger than \(|RP_{kL-j+1}|\), apart from some error term.
\begin{lemma}\label{expansion rate}
For each \(j = 1, \dots, L/2 - 1\), we have
\[
\EE[|RP_{kL-j-1}|\mid \cE_{k-1}] \geq \left(1 + \frac{\varepsilon}{528}\right) \cdot \EE[|RP_{kL-j+1}|\mid \cE_{k-1}] - 15 \cdot 4^k \cdot \frac{\varepsilon}{528} \cdot n \cdot \exp\left(-\frac{\varepsilon L}{176}\right).
\]
\end{lemma}

Now we are ready to show Lemma \ref{iterate} assuming Lemma \ref{expansion rate} is true. Note that
\begin{align*}
n \cdot \exp\left(-\frac{\varepsilon L}{176}\right) &\leq n \cdot \exp\left(-\frac{10^{k-1}}{2 \log n} \cdot \frac{10^5 \log n}{176}\right) \leq n \cdot \exp\left(-\frac{1}{2} \cdot 3 \cdot 10^k\right) \\
&\leq n \cdot \exp\left(-\frac{1}{2} \cdot 10^k\right) \cdot e^{-10k} \leq \frac{f(k)}{60 \cdot 16^k}.
\end{align*}

\noindent Claim \ref{assumption} (ii) implies
\[
15 \cdot 4^k \cdot \frac{\varepsilon}{528} \cdot n \cdot \exp\left(-\frac{\varepsilon L}{176}\right) \leq \frac{\varepsilon}{2112} \cdot \frac{f(k)}{4^k} \leq \frac{\varepsilon}{2112} \cdot \EE[|RP_{kL}|\mid \cE_{k-1}] \leq \frac{\varepsilon}{2112} \cdot \EE[|RP_{kL-j+1}|\mid \cE_{k-1}]
\]
for each \( j = 1, \dots, L/2 - 1 \). Thus, the conclusion of Lemma \ref{expansion rate} implies
\[
\EE[|RP_{kL-j-1}|\mid \cE_{k-1}] \geq \left(1 + \frac{\varepsilon}{704}\right) \cdot \EE[|RP_{kL-j+1}|\mid \cE_{k-1}]
\]
for each \( j = 1, \dots, L/2 - 1 \). By applying this iteratively to \( j = 1, 3, 5, \dots, L/2 - 1 \), we can deduce that
\[
\EE[|RP_{kL-L/2}|\mid \cE_{k-1}] \geq \left(1 + \frac{\varepsilon}{704}\right)^{L/4} \cdot \EE[|RP_{kL}|\mid \cE_{k-1}] \geq \left(1 + \frac{\varepsilon}{704}\right)^{L/4} \cdot \frac{f(k)}{4^k}.
\]
Note that for all real numbers \( 0 \leq t \leq 1 \), we have \( 1 + t \geq \exp(t/2) \). Hence, we have
\begin{align*}
\left(1 + \frac{\varepsilon}{704}\right)^{L/4} &\geq \exp\left(\frac{\varepsilon}{1408} \cdot \frac{L}{4}\right) \geq \exp\left(\frac{10^{k-1}}{2\log n} \cdot \frac{10^5 \log n}{4 \cdot 1408}\right) \geq \exp\left(\frac{1}{2} \cdot 10^k\right) \cdot \frac{3}{2} \cdot 4^k.
\end{align*}
So, we obtain
\begin{align*}
\EE[|RP_{kL-L/2}|\mid \cE_{k-1}] &\geq \left(1 + \frac{\varepsilon}{704}\right)^{L/4} \cdot \frac{f(k)}{4^k} \geq \exp\left(\frac{1}{2} \cdot 10^k\right) \cdot \frac{3}{2} \cdot 4^k \cdot \frac{f(k)}{4^k} \\
&= \frac{3}{2} \cdot \exp\left(\frac{1}{2} \cdot 10^k\right) \cdot f(k) = \frac{3}{2}n.
\end{align*}
But since always \( |RP_{kL-L/2}| \leq |V(G)| \leq n \), we must clearly have \( \EE[|RP_{kL-L/2}|\mid \cE_{k-1}] \leq n \). \medskip \hfill $\blacksquare$

The proof of Lemma \ref{expansion rate} relies on the following key lemma. In fact, this lemma is equivalent to the definition of $G$ being a robust sublinear expander. Moreover, it is a robust version of Lemma 6.1 in \cite{Alon2025}. Note that forbidding a color of some edge is equivalent to forbidding the vertices incident with this edge. Therefore, in the following, we only need to forbid the vertices. Let $\phi_0$ be the set of forbidden vertices with size $O(t^2\log n \cdot \log\log n)<\varepsilon \cdot |U|/40$, where $|U|>O(t^2\log^3 n)$.
\begin{lemma} \label{core}
Let $G$ be a robust sublinear expander. Let \( U \subseteq V(G) \) be a vertex subset of size \( |U| \leq n^{1-\varepsilon} \), where $\varepsilon$ is defined in \eqref{epi}. Let $|\phi_0|\leq  \varepsilon \cdot |U|/40$. Consider a (not necessarily proper) coloring of the edges of \( G \) between \( U \) and \( \overline{U}\setminus \phi_0 \) with the colors red and blue. Then, at least one of the following two statements holds.
\begin{itemize}
    \item[(i)] we can find a subset \( F_{\text{red}} \subseteq E(G) \) of red edges in \( G \) with \( |F_{\text{red}}| \geq (\varepsilon/10) \cdot d(G) \cdot |U| \) and \(\deg_{F_{\text{red}}}(v) \leq  \lceil d(G)\rceil \) for all \( v \in U \), or
    \item[(ii)] we can find a subset \( F_{\text{blue}} \subseteq E(G) \) of blue edges in \( G \) with \( |F_{\text{blue}}| \geq (\varepsilon/10) \cdot d(G) \cdot |U| \) and \(\deg_{F_{\text{blue}}}(v') \leq \lceil d(G)\rceil \) for all \( v' \in  \overline{U}\setminus \phi_0 \).
\end{itemize}
\end{lemma}

\begin{proof}
Let \( E_{\text{red}} \subseteq E(G) \) and \( E_{\text{blue}} \subseteq E(G) \) be the set of red and blue edges in \( G \), respectively. Define 
$F'_{\text{red}}:=\{(v, v') \in E_{\text{red}}:\deg_{E_{\text{red}}}(v) \leq d(G)\}$ and $F'_{\text{blue}}:=\{(v, v') \in E_{\text{blue}}:\deg_{E_{\text{blue}}}(v) \leq d(G)\}$. If \( |F'_{\text{red}}| \geq (\varepsilon/10) \cdot d(G) \cdot |U| \) or \( |F'_{\text{blue}}| \geq (\varepsilon/10) \cdot d(G) \cdot |U| \), then we are done. Hence, we may assume that \( |F'_{\text{red}}| < (\varepsilon/10) \cdot d(G) \cdot |U| \) and \( |F'_{\text{blue}}| < (\varepsilon/10) \cdot d(G) \cdot |U| \), and therefore, \( |F'_{\text{red}} \cup F'_{\text{blue}}| < (\varepsilon/5) \cdot d(G) \cdot |U| \).

Let $ U_{\text{red}}:=\{v \in U :\deg_{E_{\text{red}}}(v) \geq d(G)\}$ and $V_{\text{blue}}:=\{v' \in  \overline{U}\setminus \phi_0: \deg_{E_{\text{blue}}}(v') \geq d(G)\}$. Note that every edge \((v, v') \in E(G)\) with \( v \in U \setminus U_{\text{red}} \) and \( v' \in V(G) \setminus (U \cup V_{\text{blue}}\cup \phi_0) \) is contained in \( F'_{\text{red}} \) or \( F'_{\text{blue}} \). If \( |U_{\text{red}}| \geq (\varepsilon/10) \cdot |U| \), then we can choose a set \( F_{\text{red}} \) that satisfies (i) by choosing \(  \lceil d(G) \rceil \) edges \((v, v') \in E_{\text{red}}\) for each \( v \in U_{\text{red}}\). Similarly, if \( |V_{\text{blue}}| \geq (\varepsilon/10) \cdot |U| \), then we can choose a set \( F_{\text{blue}} \) that satisfies (ii) by choosing \( \lceil d(G) \rceil \) edges \((v, v') \in E_{\text{blue}}\) for each \( v' \in V_{\text{blue}}\). Hence, we may assume that \( |U_{\text{red}}| < (\varepsilon/10) \cdot |U| \) and \( |V_{\text{blue}}| < (\varepsilon/10) \cdot |U| \), and therefore, \( |U_{\text{red}} \cup V_{\text{blue}}| < (\varepsilon/5) \cdot |U|\).

Let \( U' = U \setminus U_{\text{red}} \), then \( |U'| \geq (1 - \varepsilon/10) \cdot |U| \geq (9/10) \cdot |U| \). Hence, setting \( F = F'_{\text{red}} \cup F'_{\text{blue}} \), we have \( N_{G-F}(U')\setminus \phi_0 \subseteq U_{\text{red}} \cup V_{\text{blue}} \) and hence \( |N_{G-F}(U')\setminus \phi_0| \leq |U_{\text{red}} \cup V_{\text{blue}}| < (\varepsilon/5) \cdot |U| \) and so $|N_{G-F}(U')| \leq (\varepsilon/5)\cdot |U|+(\varepsilon/20)\cdot |U| \leq (\varepsilon/4) \cdot |U'|$. On the other hand, \( |U'| \leq |U| \leq n^{1-\varepsilon} \) and \( |F| = |F'_{\text{red}} \cup F'_{\text{blue}}| < (\varepsilon/5) \cdot d(G) \cdot |U| <(\varepsilon/4) \cdot d(G) \cdot |U'| \), so this is a contradiction to \( G \) being a robust sublinear expander.
\end{proof}
We also need a lemma in \cite{Alon2025} concerning certain events assigned to the edges of some bipartite graph.
\begin{lemma}[\cite{Alon2025}]\label{bip prob}
Let \( 0 \leq \varepsilon \leq 1 \) and let \( 1 \leq \ell \leq d \) be integers, and consider a bipartite graph \( G \) with vertex set \( U \) on one side and vertex set \( W \) on the other side. Let us assume that \( G \) has at least \( \varepsilon d |U| \) edges, but that every vertex \( w \in W \) satisfies \(\deg(w) \leq d\). For every edge \( e \in E(G) \), consider an event \( \mathcal{E}_e \) which holds with probability at least \( \ell /d \). Suppose that for every vertex \( v \in U \cup W \), the events \( \mathcal{E}_e \) are mutually independent for all edges \( e \) incident with \( v \). Then, with probability at least \( 1 - 15 \exp(-\varepsilon \ell /16) \), there exists a subset \( E' \subseteq E(G) \) of size \( |E'| \geq (\varepsilon /12) \cdot \ell \cdot |U| \) such that \( \mathcal{E}_e \) holds for each \( e \in E' \) and such that for every vertex \( w \in W \), we have \(\deg_{E'}(w) \leq 2\ell\).
\end{lemma}
\medskip
\indent For convenience, denote $RP_{\phi_1}(v_i,A_{l} \setminus \gamma(v, v'))$ by $RP(A_{l} \setminus \gamma(v, v'))$ for $l\in [T]$. Now we define two types of edge sets. 
\begin{itemize}
    \item[(i)] An edge set \( F \subseteq E(G) \) is called \textit{type-I} if 
\[
F\subseteq \{(v,v')\in E(G):v \in RP(A_{kL-j} \setminus \gamma(v, v')),v' \notin (RP_{kL-j} \cup\phi_0)\}
\]
with \( |F| \geq (\varepsilon / 10) \cdot d(G) \cdot |RP_{kL-j}| \) and \( \deg_F(v') \leq \lceil d(G)\rceil \) for all \( v' \notin (RP_{kL-j} \cup\phi_0)\).
    \item[(ii)] An edge set \( F \subseteq E(G) \) is called \textit{type-II} if 
    \[
F\subseteq \{(v,v')\in E(G): v \in RP(A_{kL-j} \setminus \gamma(v, v')),v' \notin (RP_{kL-j} \cup\phi_0),\gamma(v, v') \in A_{(k-1)L}   \}
\]
with \( |F| \geq (\varepsilon / 132) \cdot L \cdot |RP_{kL-j}| \) and \( \deg_{F}(v') \leq 2L \) for all \( v' \notin (RP_{kL-j} \cup\phi_0) \). 
\end{itemize}
We say that the event \( \mathcal F \) holds if there exists a type-I set $F$, but no type-II set $F'$ exists.\medskip

We will deduce the following three lemmas, which play an important role in the proof of Lemma \ref{expansion rate}. Let \( \overline{\mathcal F} \) denote the complementary event to \(\mathcal F \). The first lemma states that $\mathcal F$ is very unlikely. The next two lemmas establish the expansion rate of the thinning process conditioned on different events. 
\begin{lemma}\label{c1}
\(\PP[\mathcal F] \leq 15 \exp \left( -\frac{\varepsilon L}{176} \right).\)
\end{lemma}
\begin{proof}
It suffices to prove $\mathbb{P}[\mathcal{F} \mid A_{kL-j}] \leq 15 \exp(-\varepsilon L/176)$ for every possible outcome of $A_{kL-j}$. So, let us fix any outcome of $A_{kL-j}$. We may assume that there exists a type-I set $F \subseteq E(G)$. For otherwise, if $A_{kL-j}$ does not admit any type-I set, then $\mathbb{P}[\mathcal{F} \mid A_{kL-j}] = 0$. Let $U = RP_{kL-j}$, then 
\[|F| \geq (\varepsilon/10) \cdot d(G) \cdot |U| \geq (\varepsilon/11) \cdot \lceil d(G) \rceil \cdot |U|,\]
and every edge in $F$ is of the form $(v, v')$ with $v \in U$ and $v' \in V(G) \setminus (U \cup \phi_0)$ and $v \in RP(A_{kL-j} \setminus \{ \gamma(v,v') \}) $. Furthermore, we have $\deg_F(v') \leq \lceil d(G) \rceil $ for all $v' \in V(G) \setminus (U \cup \phi_0)$.

We can now apply Lemma \ref{bip prob} to the bipartite graph with edge set $F$ between the vertex sets $U$ and $V(G) \setminus (U \cup \phi_0)$, where for each edge $(v, v') \in F$, we consider the event that $\gamma(v,v') \in A_{(k-1)L}$. These events are independent for edges of different colors, in particular, for edges at the same vertex. Furthermore, for every edge $(v, v') \in F$, conditioned on our fixed outcome of $A_{kL-j}$, by Lemma \ref{cond prob} (i) (applied with $p = 1 - 1/T$), and using the fact that $j \leq L/2$ and $d(G) \geq 12T$, we have 
\[
\PP[\gamma(v,v') \in A_{(k-1)L}]\geq \frac{L-j}{6T}\geq \frac{L}{12T}\geq \frac{L}{\lceil d(G) \rceil}.
\]

Lemma \ref{bip prob} now shows that with probability at least $1-15 \exp(-(\varepsilon/11)L/16) = 1-15 \exp(-\varepsilon L/176)$ between the vertex sets $U$ and $V(G) \setminus (U \cup \phi_0)$, there exists a subset $F' \subseteq F \subseteq E(G)$ of size $|F'| \geq (\varepsilon/11) \cdot L \cdot |U|/12 =(\varepsilon/132) \cdot L \cdot |U|$ such that $\gamma(v,v') \in A_{(k-1)L}$ for all $(v,v') \in F'$ and $\deg_{F'}(v') \leq 2L$ for all $v' \in V(G) \setminus (U \cup \phi_0)$. Hence, $F'$ is a type-II set. Note that $\mathcal{F}$ does not hold if such a set $F'$ exists. This shows that 
\[
\PP[\mathcal{F}]\leq 1-\PP[F'\text{ exists}]\leq 15 \exp(-\varepsilon L/176),
\]
as desired.
\end{proof}
\begin{lemma}\label{c2}
 We have \(\PP[\cE_{k-1} \text{ and } \overline{\mathcal F}] > 0\) and
\[
\EE \left[ |RP_{kL-j-1}| - \left( 1 + \frac{\varepsilon}{528} \right) \cdot |RP_{kL-j+1}| \bigg| \cE_{k-1}, \overline{\mathcal F} \right] \geq 0.
\]
\end{lemma}
\begin{proof}
By Claim \ref{assumption} (i) and Lemma \ref{c1}, we have
\[
\mathbb{P}[\mathcal{F}] \leq 15 \exp\left(-\frac{\varepsilon L}{176}\right) \leq 15 \exp\left(-\frac{10^{k-1}}{2 \cdot \log n} \cdot \frac{10^5 \log n}{176}\right) \leq 15 e^{-10k} < 1/4^k<\mathbb{P}[\cE_{k-1}] .
\]
Hence, $\mathbb{P}[\cE_{k-1} \text{ and } \overline{\mathcal{F}}] > 0$. 

For the second part, it suffices to prove that
\[
\mathbb{E} \left[ |RP_{kL-j-1}| - \left(1 + \frac{\varepsilon}{528}\right) \cdot |RP_{kL-j+1}| \bigg| A_{(k-1)L}, A_{kL-j} \right] \geq 0
\]
for every possible outcome of $A_{(k-1)L}$ and $A_{kL-j}$ satisfying $\cE_{k-1}$ and $\overline{\mathcal{F}}$. So, let us fix $A_{(k-1)L} \supseteq A_{kL-j}$ satisfying $\cE_{k-1}$ and $\overline{\mathcal{F}}$. Let $U = RP_{kL-j}$. By Claim \ref{assumption} (i), we have $|U| \leq |RP_{(k-1)L}| < f(k-1) = n^{1-\varepsilon}$. 

We now construct an auxiliary red-blue edge-coloring on the bipartite graph between $U$ and $V(G) \setminus (U \cup \phi_0)$. For an edge $(v, v') \in E(G)$ with $v \in U$ and $v' \in V(G) \setminus (U \cup \phi_0)$, let us color $(v, v')$ 
\begin{itemize}
    \item[(i)] red if $v \notin RP(A_{kL-j} \setminus \{\gamma(v, v')\}) $;
    \item[(ii)] blue otherwise.
\end{itemize}
Now, at least one of the two alternatives in Lemma \ref{core} must hold, and we split into cases accordingly.\medskip 

\noindent \textbf{Case 1. Lemma \ref{core} (i) holds.} In this case, there exists a subset $F_{\text{red}} \subseteq E(G)$ of red edges in $G$ with $|F_{\text{red}}| \geq (\varepsilon /10) \cdot d(G) \cdot |U|$ and $\deg_{F_{\text{red}}}(v) \leq \lceil d(G)\rceil$ for all $v \in U$.

Note that $\PP[\gamma(v,v^{\prime}) \notin A_{kL-j+1}\mid A_{(k-1)L}, A_{kL-j}]  ]\geq 1/T \geq 1/\lceil d(G)\rceil$. Furthermore, these events $\gamma(v,v^{\prime}) \notin A_{kL-j+1}$ are independent for edges $(v,v^{\prime})$ of different colors, even after conditioning on $A_{(k-1)L}$ and $A_{kL-j}$. Thus, for each $v \in U$, we can apply Lemma \ref{cond prob} to the events $\gamma(v,v^{\prime}) \notin A_{kL-j+1}$ for the $\deg_{F_{\text{red}}}(v) \leq \lceil d(G)\rceil$ different edges $(v,v^{\prime}) \in F_{\text{red}}$, yielding that
\begin{equation} \label{case1pro}
 \mathbb{P}[\gamma(v,v^{\prime}) \notin A_{kL-j+1} \text{ for some edge } (v,v^{\prime}) \in F_{\text{red}} \text{ at } v \mid A_{(k-1)L}, A_{kL-j}] \geq \frac{\deg_{F_{\text{red}}}(v)}{2\lceil d(G)\rceil}.   
\end{equation}

\begin{claim}\label{cred}
If $(v,v^{\prime}) \in F_{\text{red}}$ satisfies $\gamma(v,v^{\prime}) \notin A_{kL-j+1}$, then $v \notin (RP_{kL-j+1} \cup \phi_0)$.
\end{claim}
Indeed, since $\gamma(v,v^{\prime}) \notin A_{kL-j+1}$, we have $A_{kL-j+1} \subseteq A_{kL-j} \setminus \{\gamma(v,v^{\prime})\}$ and so $RP_{kL-j+1} \subseteq RP(A_{kL-j} \setminus \{\gamma(v,v^{\prime})\}) $. The red coloring of $(v,v^{\prime})$ means $v \notin RP(A_{kL-j} \setminus \{\gamma(v,v^{\prime})\})$, which yields $v \notin (RP_{kL-j+1} \cup \phi_0)$. \hfill $\blacksquare$

\noindent Thus, by Claim \ref{cred} and \eqref{case1pro}, we have
\[
\mathbb{P}[v \notin (RP_{kL-j+1} \cup \phi_0) \mid A_{(k-1)L}, A_{kL-j}] \geq \frac{\deg_{F_{\text{red}}}(v)}{2\lceil d(G)\rceil},
\]
for every $v \in U$. Summing this up for all $v \in U$ yields
\begin{align*}
\mathbb{E}[|U \setminus RP_{kL-j+1}| \mid A_{(k-1)L}, A_{kL-j}] &= \sum_{v \in U} \mathbb{P}[v \notin (RP_{kL-j+1} \cup \phi_0) \mid A_{(k-1)L}, A_{kL-j}] \\
&\geq \sum_{v \in U} \frac{\deg_{F_{\text{red}}}(v)}{2\lceil d(G)\rceil} = \frac{|F_{\text{red}}|}{2\lceil d(G)\rceil} \geq \frac{(\varepsilon /10) \cdot d(G) \cdot |U|}{2\lceil d(G)\rceil} \geq \frac{\varepsilon}{21} \cdot |U|,
\end{align*}
where the final inequality uses $\lceil d(G)\rceil \leq 21/20 \cdot d(G)$ for sufficiently large $n$. Since $RP_{kL-j+1} \subseteq U$, we conclude
\[
\mathbb{E}[|RP_{kL-j+1}| \mid A_{(k-1)L}, A_{kL-j}] = |U| - \mathbb{E}[|U \setminus RP_{kL-j+1}| \mid A_{(k-1)L}, A_{kL-j}] \leq \left(1 - \frac{\varepsilon}{21}\right) \cdot |U|.
\]
On the other hand, we always have $RP_{kL-j-1} \supseteq U$, and hence the same holds in expectation. Thus,
\[
\mathbb{E}[|RP_{kL-j-1}| - \left(1 + \frac{\varepsilon}{528}\right) \cdot |RP_{kL-j+1}| \mid A_{(k-1)L}, A_{kL-j}] \geq |U| - \left(1 + \frac{\varepsilon}{528}\right) \cdot \left(1 - \frac{\varepsilon}{21}\right) \cdot |U| \geq 0,
\]
as desired.
\medskip \\
\noindent  \textbf{Case 2. Lemma \ref{core} (ii) holds.} In this case, there exists a subset $F_{\text{blue}} \subseteq E(G)$ of blue edges in $G$ with $|F_{\text{blue}}| \geq (\varepsilon/10) \cdot d(G) \cdot |U|$ and $\deg_{F_{\text{blue}}}(v') \leq \lceil d(G)\rceil$ for all $v' \in V(G) \setminus (U \cup \phi_0)$, which means that $F_{\text{blue}}$ is type-I. Since we choose outcomes of $A_{(k-1)L}$ and $A_{kL-j}$ that do not satisfy $\mathcal{F}$, this means that there exists a set $F' \subseteq E(G)$ of type-II with $|F'| \geq (\varepsilon/132) \cdot L \cdot |U|$ and $\deg_{F'}(v') \leq 2L$ for all $v' \in V(G) \setminus (U \cup \phi_0)$, such that each edge in $F'$ is of the form $(v,v')$ with $v \in RP(A_{kL-j} \setminus \{\gamma(v,v')\}) \subseteq U$ and $v' \notin (U \cup\phi_0)$ and satisfies $\gamma(v,v') \in A_{(k-1)L}$.

We claim that for every $(v,v') \in F'$, the probability (conditioning on $A_{(k-1)L}$ and $A_{kL-j}$) of having $\gamma(v,v') \in A_{kL-j-1}$ is at least $1/(2L)$. Indeed, since $(v,v') \in F'$, we have $\gamma(v,v') \in A_{(k-1)L}$. So, by Lemma \ref{cond prob} (ii) (applied with $p = 1 - 1/T$), the probability of having $\gamma(v,v') \in A_{kL-j-1}$ is at least $(T/(L-j)) \cdot 1/(2T) \geq 1/(2L)$. Furthermore, these events $\gamma(v,v') \in A_{kL-j-1}$ are independent for edges $(v,v')$ of different colors. Now, for each $v' \in V(G) \setminus (U \cup \phi_0)$, we can apply Lemma \ref{union} to the events $\gamma(v,v') \in A_{kL-j-1}$ for the $\deg_{F'}(v') \leq 2L$ different edges $(v,v') \in F'$. Thus, for every $v' \in V(G) \setminus (U \cup \phi_0)$, we have
\begin{equation} \label{case2pro}
\mathbb{P}[\gamma(v,v') \in A_{kL-j-1} \text{ for some edge } (v,v') \in F' \text{ at } v' | A_{(k-1)L}, A_{kL-j}] \geq \frac{\deg_{F'}(v')}{4L}.
\end{equation}

\begin{claim}\label{cblue}
If $(v,v') \in F'$ satisfies $\gamma(v,v') \in A_{kL-j-1}$, then $v' \in RP_{kL-j-1}$.
\end{claim}
Indeed, as $(v,v') \in F'$, we have $v \in RP(A_{kL-j} \setminus \{\gamma(v,v')\}) $. This means that there exists a rainbow walk from $x$ to $v$ with colors in $A_{kL-j} \setminus \{\gamma(v,v')\} \subseteq A_{kL-j-1} \setminus \{\gamma(v,v')\}$. Adding the edge $(v,v')$ gives a rainbow walk from $x$ to $v'$ with colors in $A_{kL-j-1}$ (since $\gamma(v,v') \in A_{kL-j-1}$). Thus, $v' \in RP_{kL-j-1}$. \hfill $\blacksquare$

\noindent Thus, by Claim \ref{cblue} and \eqref{case2pro}, we have
\[
\mathbb{P}[v' \in RP_{kL-j-1} | A_{(k-1)L}, A_{kL-j}] \geq \frac{\deg_{F'}(v')}{4L},
\]
for every $v' \in V(G) \setminus (U \cup \phi_0)$. Summing this up for all $v' \in V(G) \setminus (U \cup \phi_0)$ yields
\begin{align*}
\mathbb{E}[|RP_{kL-j-1} \setminus U| | A_{(k-1)L}, A_{kL-j}] &= \sum_{v' \in V(G) \setminus U} \mathbb{P}[v' \in U_{kL-j-1} | A_{(k-1)L}, A_{kL-j}] \\
&\geq \sum_{v' \in V(G) \setminus U} \frac{\deg_{F'}(v')}{4L} = \frac{|F'|}{4L} \geq \frac{(\varepsilon/132) \cdot L \cdot |U|}{4L} \geq \frac{\varepsilon}{528} \cdot |U|.
\end{align*}

\noindent Since $RP_{kL-j-1} \supseteq  U$, we have
\[
\mathbb{E}[|RP_{kL-j-1}| | A_{(k-1)L}, A_{kL-j}] = |U| + \mathbb{E}[|RP_{kL-j-1} \setminus U| | A_{(k-1)L}, A_{kL-j}] \geq \left(1 + \frac{\varepsilon}{528}\right) \cdot |U|.
\]
\noindent On the other hand, we always have $RP_{kL-j+1} \subseteq  U$, and so $\mathbb{E}[|RP_{kL-j+1}| | A_{(k-1)L}, A_{kL-j}] \leq |U|$. Thus,
\[
\mathbb{E}[|RP_{kL-j-1}| - \left(1 + \frac{\varepsilon}{528}\right) \cdot |RP_{kL-j+1}| | A_{(k-1)L}, A_{kL-j}] \geq \left(1 + \frac{\varepsilon}{528}\right) \cdot |U| - \left(1 + \frac{\varepsilon}{528}\right) \cdot |U| \geq 0,
\]
as desired. 
\end{proof}
\begin{lemma} \label{c3}
If \(\PP[\cE_{k-1} \text{ and } \mathcal F] > 0\), then
\[
\EE \left[ |RP_{kL-j-1}| - \left( 1 + \frac{\varepsilon}{528} \right) \cdot |RP_{kL-j+1}| \bigg| \cE_{k-1}, \mathcal F \right] \geq -\frac{\varepsilon}{528} \cdot n.
\]
\end{lemma}
\begin{proof}
Note that
\[
|RP_{kL-j-1}| - \left(1 + \frac{\varepsilon}{528}\right) \cdot |RP_{kL-j+1}| \geq -\frac{\varepsilon}{528}|RP_{kL-j+1}|\geq -\frac{\varepsilon}{528}n.
\]
Since \(\PP[\cE_{k-1} \text{ and } \mathcal F] > 0\), the conditional expectation is well defined and so we complete our proof.
\end{proof}
Let us show how these lemmas imply Lemma \ref{expansion rate}.
\begin{proof}[Proof of Lemma \ref{expansion rate} assuming Lemmas \ref{c1},\ref{c2} and \ref{c3}]
If \(\PP[\cE_{k-1} \text{ and } \mathcal F] > 0\), then we have
\begin{align*}
&\EE \left[ |RP_{kL-j-1}| - \left( 1 + \frac{\varepsilon}{528} \right) \cdot |RP_{kL-j+1}| \bigg| \cE_{k-1} \right] \\
&= \EE \left[ |RP_{kL-j-1}| - \left( 1 + \frac{\varepsilon}{528} \right) \cdot |RP_{kL-j+1}| \bigg| \cE_{k-1}, \mathcal F \right] \cdot \PP[\mathcal F \mid \cE_{k-1}] \\
&\quad + \EE \left[ |RP_{kL-j-1}| - \left( 1 + \frac{\varepsilon}{528} \right) \cdot |RP_{kL-j+1}| \bigg| \cE_{k-1}, \overline{\mathcal F} \right] \cdot \PP[\overline{\mathcal F} \mid \cE_{k-1}] \\
&\geq -\frac{\varepsilon}{528} \cdot n \cdot \PP[\mathcal F \mid \cE_{k-1}] + 0 \cdot \PP[\overline{\mathcal F} \mid \cE_{k-1}] \geq -\frac{\varepsilon}{528} \cdot n \cdot \frac{\PP[\mathcal F]}{\PP[\cE_{k-1}]} \geq -15 \cdot 4^k \cdot \frac{\varepsilon}{528} \cdot n \cdot \exp \left( -\frac{\varepsilon L}{176} \right).
\end{align*}
If \(\PP[\cE_{k-1} \text{ and } \mathcal F] = 0\), then we have
\[
\EE \left[ |RP_{kL-j-1}| - \left( 1 + \frac{\varepsilon}{528} \right) \cdot |RP_{kL-j+1}| \bigg| \cE_{k-1} \right] = \EE \left[ |RP_{kL-j-1}| - \left( 1 + \frac{\varepsilon}{528} \right) \cdot |RP_{kL-j+1}| \bigg| \cE_{k-1}, \overline{\mathcal F} \right] \geq 0.
\]
So, in both cases, we obtain
\[
\EE \left[ |RP_{kL-j-1}| - \left( 1 + \frac{\varepsilon}{528} \right) \cdot |RP_{kL-j+1}| \bigg| \cE_{k-1} \right] \geq -15 \cdot 4^k \cdot \frac{\varepsilon}{528} \cdot n \cdot \exp \left( -\frac{\varepsilon L}{176}\right),
\]
as desired.
\end{proof}

\section{Applications} \label{s7}
In this section, we discuss several applications of Theorem \ref{main result}.
\subsection{Additive combinatorics}
\textbf{Small doubling.} We prove Theorem \ref{small-doubling} here. The asymmetric variant of Theorem \ref{small-doubling} can be seen in \cite{Alon2025}.
\begin{proof}[Proof of Theorem \ref{small-doubling}]
Let \( A \subseteq \Gamma \) with \( |A \cdot A| \leq K|A| \) and let \( S \) be a dissociated subset of \( A \) of maximum size. Then \( |S| = \dim A \). We will show that \( |S| \leq O(\log^{1+o(1)} |A|) \). Define a bipartite graph \( H \) with vertex partitions \( A \) and \( B = A \cdot S = \{as : a \in A, s \in S\} \) where \((a, b)\in A \times B\) is an edge of \( H \) if \( b = as \) for some \( s \in S \). Also, assign color \( s \) to this edge. This gives a proper edge-coloring of \( H \).

Note that \( |V(H)| \leq |A| + |A \cdot A| \leq O(|A|)\), and \(e( H)= |A||S|\). Suppose, for contradiction, that \( |S| > C_K \log |A| (\log \log |A|)^6 \) for a sufficiently large constant \( C_K \) depending on \( K \). Then we can use Theorem \ref{main result} to find a rainbow clique subdivision, and in particular, a rainbow cycle in \( G \). Since \( H \) is bipartite, any cycle has even length. A rainbow cycle with distinct colors \( s_1, s_2, \ldots, s_{2\ell} \) appearing in that order implies
$s_1 s_2^{-1} s_3 s_4^{-1} \cdots s_{2\ell-1} s_{2\ell}^{-1} = e$. This contradicts that \( S \) is dissociated and shows that \( |S| \leq  O(\log^{1+o(1)} |A|) \).
\end{proof} 
\noindent \textbf{Large convolution.} Now we prove Theorem \ref{convolution}.
\begin{proof}[Proof of Theorem \ref{convolution}]
Assume $|A|\geq |B|$. Let \( L \subseteq S \) be a largest dissociated subset of \( S \), so that \( |L| = \dim(S) \). Consider a bipartite graph \( G = (V, E) \) with parts \( A \) and \( B \), where each edge is colored by an element of \( L \). Specifically, a vertex \( a \in A \) is connected to a vertex \( b \in B \) by an edge colored \( \lambda \in L \) if and only if 
$a - b = \lambda$. By definition of \( S \), for each \( \lambda \in L \subseteq S \), we have
$(A * (-B))(\lambda) \geq \sigma$, so there are at least \( \sigma \) pairs \( (a, b) \in A \times B \) with \( a - b = \lambda \). Therefore, the total number of edges satisfies $|E| \geq \sigma |L|$. Note that
$|V| = |A| + |B| \leq 2|A|$.

Suppose $|L| > C \cdot |A| \cdot \sigma^{-1} \cdot \log |A|\cdot (\log \log |A|)^6$, for a sufficiently large constant \( C \). Then the average degree of \( G \) is at least $C \log |A|\cdot (\log \log |A|)^6$.
By Theorem \ref{main result}, \(  G \) contains a rainbow clique subdivision, and thus a rainbow cycle. Now we can derive a contradiction similar to the proof of Theorem \ref{small-doubling}. Therefore, we have
$|L| \le |A| \cdot \sigma^{-1} \cdot O(\log^{1+o(1)} |A|)$.
\end{proof}

\subsection{Number theory}
In this subsection, we prove Theorem \ref{Bhg}.
\begin{proof}[Proof of Theorem \ref{Bhg}]
Let us consider a bipartite graph on $V_1 \cup V_2$ where $V_i$ is a copy of $\mathbb{Z}_n$ for $i \in [2]$. For every $b\in B$ and every $s\in \mathbb{Z}_n$, add an edge between vertex $s \in V_1$ and vertex $s+b\in V_2$ and color this edge with color $b$. This way, we obtain an $|B|$-regular graph with a proper edge-coloring (with $|B|$ colors). If $G$ contains no rainbow $TK_t$ for every $t$, then by Corollary \ref{coro1}, we have $|B|=d(G)\leq (\log n)^{1+o(1)}$.

Now suppose $G$ contains a rainbow $TK_t$, for some $t\geq 3$. 
Fix any rainbow cycle, let $b_1,\ldots, b_m\in B$ be the colors on the edges of the cycle in the order they appear on the cycle, starting at some vertex $s\in V_1$. Note that $m$ is even (as the graph is bipartite) and the elements $b_1,\ldots, b_m\in B$ are distinct (as the cycle is rainbow). We must have $s=s+b_1-b_2+b_3-b_4+\ldots+b_{m-1}-b_m$. Thus, we obtain $b_1+b_3+\ldots+b_{m-1}=b_2+b_4+\ldots +b_m$, where $m\leq \log n \log\log n$. Hence, we construct a set $B'=\{b_1, ..., b_m\}$, which is not a $B_{m/2}[1]$-set. This completes our proof.
\end{proof}

\section*{Acknowledgements}
We thank Ruonan Li, Minghui Ouyang, Tianchi Yang and Chi Hoi Yip for helpful discussions.

\end{document}